\documentclass[10pt,oneside]{amsart}

\usepackage{amssymb}
\usepackage{formel}
\usepackage{diagram}
\usepackage[enableskew]{youngtab}

\theoremstyle{plain}
\newtheorem{theorem}{Theorem}[section]
\newtheorem{prop}[theorem]{Proposition}
\newtheorem{lemma}[theorem]{Lemma}
\newtheorem{corollary}[theorem]{Corollary}
      
\theoremstyle{definition}
\newtheorem{definition}[theorem]{Definition}

\theoremstyle{remark}
\newtheorem{remark}[theorem]{Remark}
\newtheorem{example}[theorem]{Example}

\numberwithin{equation}{section}

\makeatletter
\def\@setcopyright{}
\def\serieslogo@{}
\makeatother

\begin{document}


\author{Tobias Pecher}
\address{Emmy-Noether-Zentrum, Department Mathematik, Bismarckstrasse 1 $\!\!\frac{1}{2}$, 91054 Erlangen, Germany}
\email{pecher@mi.uni-erlangen.de}

\title[Skew multiplicity-free Modules]{Classification of Skew Multiplicity-free Modules}

\begin{abstract}
Let $G$ be a reductive group defined over $\CC$ with a finite dimensional representation $V$. The action of $G$ is said to be skew multiplicity-free (SMF) if the exterior algebra $\bigwedge V$ contains no irreducible representation of $G$ with multiplicity $> 1$. This paper is concerned with a classification of such representations.  
\end{abstract}




\date{\today}

\maketitle

\section{Introduction}

The notion of a skew multiplicity-free (SMF) action was introduced in \cite{Ho1}, where a main aim is the interpretation of invariant theory in terms of (arbitrary) multiplicity-free actions. Classic results, like the first and second fundamental theorems of invariant theory for example, can be traced back to such actions. In this setting, SMF spaces can be regarded as a model to describe ``skew'' invariants - for instance the algebra structure of $\bigwedge(\CC^n \otimes \CC^l \oplus (\CC^n)^* \otimes \CC^m)^{\GL_n}$.

SMF spaces naturally appear as the cohomology of nilradicals $\mathfrak{n}$ of semisimple Lie algebras. It is shown in~\cite{Ko} that the cohomology groups decompose multiplicity-free under the action of the Levi subalgebra. In the cases where $\mathfrak{n}$ is abelian or a Heisenberg Lie algebra the cohomology can be computed via $\bigwedge \mathfrak{n}$. 

A distinguishing fact for SMF spaces is that the algebra of $G$-invariant differential operators $\mathcal{PD}(V)^G$ on such a $V$ admits a canonical basis. Thus it is possible to formulate analogs of the famous Capelli identity, resp. Capelli operators in this setting. Some particular results on the spectral functions on those spaces are carried out in~\cite{We}.

SMF modules are skew-symmetric counterparts of multiplicity-free spaces, i.e. those for which $S(V)$ has no multiplicities. These spaces are well-investigated: They are classified due to~\cite{Ka}, \cite{BR} and~\cite{Lea} and can be characterized as those spaces that admit a dense orbit under a Borel subgroup of $G$ (i.e. they are {\it spherical}). Furthermore, one knows that a representation of a group $G$ is multiplicity-free if and only if this holds for the identity component of $G$, see e.g.~\cite{BJLR}. Unfortunately, an analogous result for exterior algebras is wrong. For instance, the action of $\CC^*\times\SL_3\times\OG_{2n}$ on $\CC^3\otimes\CC^{2n}$ is SMF while for $\CC^*\times\SL_3\times\SO_{2n}$ this is false. 

A classification of SMF modules in the case of simple $G$ and irreducible $V$ can be found in~\cite{Ho1}. An attempt towards finding a more uniform classification argument was made in \cite{St3}. Our aim is to extend these results to arbitrary
reductive groups and arbitrary (finite dimensional) representations. However, we will restrict ourselves to the subclass of so-called {\it saturated} representations, i.e. those with a sufficient number of torus factors. Also will make no attempt to classify SMF representations of arbitrary (nonconnected) groups but we will give some important examples in this case.
Finding all SMF representations for these groups is a nontrivial task and also a classification of SMF representations for finite groups may be interesting in its own right.

We give a short overview of the article: In Section~\ref{sec:basics}, basic definitions together with a necessary condition on skew multiplicity-free modules are recalled, as well as the classification of simple SMF modules. Furthermore, we generalize a necessary condition on skew multiplicity-free modules.

Section~\ref{sec:dual} deals with centralizers of reductive group actions. Here, we give a characterization of SMF modules $V$ in terms of Clifford algebras: They are exactly those spaces for which $Cl(V\oplus V^*)^G$ is commutative (Proposition~\ref{prop:clifford}). We also introduce some well-known instances of dual reductive pairs on exterior algebras. As we learned from the referee, they play an important role in the classification of SMF modules.

The irreducible SMF representations are treated in Section~\ref{sec:irred}. These actions will be mainly traced back to an instance of a dual reductive pair together with well-known multiplicity-free decomposition rules.

Section~\ref{sec:reducible} is concerned with the reducible case. Here, there arise some technical difficulties, why we will consider only saturated actions. We prove that a representation $V_1\oplus V_2$ is SMF if and only if $V_1\oplus V_2^*$ has this property by using the characterization given in Proposition~\ref{prop:clifford}. In Proposition~\ref{prop:subgraph} we give some sort of a monotonicity-criterion for SMF spaces that reduces the amount of necessary calculations tremendously.

A discussion of nonsaturated actions is held in Section~\ref{sec:non_saturated}. Here, we do not give a real classification but rather a recipe for finding an arbitrary SMF module.

Unfortunately, the lack of geometric descriptions for SMF spaces makes some case-by-case calculations inevitable in both the irreducible and the reducible case. For a better readability this is done in Appendix~\ref{app:irred}, resp.~\ref{app:reducible}. Although much of primary calculations in this paper had been done by computer algebra systems LiE and SCHUR, all results occurring in the tables can be easily verified by hand.

An a posteriori observation of the classification is that (considering only the infinite series of SMF spaces) there is a strong correspondence to the series of usual multiplicity-free representations via the Langlands dual group. Thus, it may be worthwhile to ask if one could formulate a skew analog of the spherical criterion.

\section{Preliminaries} \label{sec:basics}

Consider a reductive group $G$ acting on a complex (finite dimensional) vector space $V$. 

\begin{definition} \label{def:SMF}
We say that $(G,V)$ is {\it skew multiplicity-free (SMF)} if the exterior algebra $\bigwedge V$ is a multiplicity-free $G$ module, i.e. if $\dim\Hom_G(\Gamma, \bigwedge V) \leq 1$ for all irreducible representations $\Gamma$ of $G$.
\end{definition}

Assume that $V = V_1\oplus\dots V_l$ is a decomposition into irreducible subspaces for $G$. We say that the action is {\it saturated}, if $G$ is of the form 

\begin{equation} \label{eqn:saturated}
G = (\CC^*)^l \times G_1 \times\dots\times G_k = (\CC^*)^l \times [G,G],
\end{equation}
where $G_i$ is a simple group and the $j$th torus factor acts by scalar multiplication on $V_j$ and trivially on all other irreducible summands $V_i$.

\begin{remark} \label{rem:saturated}
Note that for a saturated $G$ action on $V$ the SMF property is equivalent to the fact that for all integers $0 \leq r_j \leq \dim(V_j)$ the subspaces

\[ \bigwedge^{r_1} V_1 \otimes \dots \otimes \bigwedge^{r_l} V_l \]
of $\bigwedge V$ are multiplicity-free as a $[G,G]$ module. In particular, if $V$ is irreducible this means that it suffices to show that each power $\bigwedge^k V$ is multiplicity-free for the semisimple part of $G$. 

Obviously, a SMF module $(G,V)$ may lose this property after a removal of torus factors of $G$. In fact, if $G$ is connected semisimple then $(G,V)$ is never SMF since both $\bigwedge^0 V$ and $\bigwedge^{\dim(V)}V$ give the trivial representation of $G = [G,G]$.
\end{remark}

\begin{remark}
Our classification actually concerns pairs $(G,V)$ with $G$ semisimple and $V$ finite dimensional, such that $G \subseteq \GL(V)$ and $G$ acting multiplicity-free on $\bigwedge V$. If $G' \rightarrow G$ is a surjective homomorphism, then the representation $(G', V)$ is also SMF. Nevertheless, the automorphism group induced by $G'$ and $G$ in $\Aut(\bigwedge V)$ are the same.
\end{remark}

\begin{definition} \label{def:geom_equiv}
Two representations $(G,\rho,V)$ and $(G',\rho',V')$ are said to be {\itshape geometrically equivalent} if there exists an isomorphism $\psi:V \stackrel{\sim}{\longrightarrow} V'$ such that for the induced isomorphism  $\GL(\psi): \GL(V) \rightarrow \GL(V')$ one has $\GL(\psi)(\rho(G))=\rho'(G')$. 

We write  $(G,\rho,V) \sim (G',\rho',V')$ for a pair of geometrically equivalent representations or, if the underlying homomorphisms are obvious, $(G,V)\sim (G',V')$.
\end{definition}

Note that $(G,V)$ and $(G',V)$ are geometrically equivalent whenever $G' \rightarrow G$ is a surjective homomorphism. 
For example, let $T$ be a maximal torus inside $G$ and $w_0$ be the longest element in the Weyl group of $G$. Then the automorphism on the character group $X(T)$ that sends $\lambda$ to $-w_0(\lambda)$ determines an automorphism $\varphi$ of $G$, such that the pullback of any representation $(G,V)$ via $\varphi$ yields the dual representation $(G,V^*)$. Moreover, any two representations of $G$ that differ by an automorphism of the Dynkin diagram of the semisimple part of $G$ are geometrically equivalent. Note also that this concept is compatible with products, i.e. for $(G,V)\sim (G',V')$ and $(H,W)\sim(H',W')$ we also have $(G\times H, V\otimes W)\sim(G'\times H',V'\otimes W')$.

By the above we can assume that $G$ is given by a product of a number of $\CC^*$-factors together with a semisimple group since any reductive group can be covered that way. 

\begin{remark}
Some notational conventions for representations: The irreducible representation of a reductive group $G$ of highest weight $\omega_1$, where $\omega_1$ is the first fundamental weight, is also denoted by $G$. (For a classical group $G \subseteq \GL_n\CC$ this is always the natural representation on $\CC^n$). Similar notations are used for plethysms, e.g $\bigwedge^2 G$, which stands for the action on the second exterior power of $G$. 
The spin representations of $\Spin_n\CC$ are denoted by $\Delta_n$. (For $n$ even $\Delta_n = \Delta_n^+ \oplus \Delta_n^-$ decomposes into the irreducible half-spin representations.)
\end{remark}

Let $G = \CC^* \times [G,G]$ with $[G,G]$ being simple and connected. We are now ready to state the classification of irreducible modules for these groups $G$. The labels in the list serve as a reference for later purposes.

\begin{theorem} \label{thm:howeslist}

Let $G$ be as above. The following list exhausts all irreducible SMF representations up to geometric equivalence.

\[ \begin{array}{|l l|l l|l l|} \hline
(A1^n) & \SL_n~(n \geq 2)           & (B1^n) & \SO_{2n+1}~(n\geq 3) & (C1^n) & \Sp_{2n}~(n\geq 2) \\
(A2^n) & S^2\SL_n~(n\geq 2)         & (B2)   & \Delta_7             & (C2)   & \bigwedge^2_0\Sp_4 \\
(Ak)   & S^k\SL_2~(k=3,\dots,6)     & (B3)   & \Delta_9             & (C3)   & \bigwedge^3_0\Sp_6 \\
(A7)   & S^3\SL_3                   & (D1^n) & \SO_{2n} ~(n \geq 4) & (G)    & G_2                \\
(A8^n) & \bigwedge^2\SL_n~(n\geq 4) & (D2)   & \Delta_{10}^+        & (E6)   & E_6                \\
(A9)   & \bigwedge^3\SL_6           & (D3)   & \Delta_{12}^+        & (E7)   & E_7                \\ \hline
\end{array} \]

\end{theorem}

\begin{proof} See~\cite[Theorem 4.7.1]{Ho1}.
\end{proof}

Note that the above remarks about geometrically equivalent representations imply that also $\GL_n$, $\bigwedge^2\GL_n$, $S^2\GL_n$, etc. are SMF. Note also that e.g. $\bigwedge^2_0\Sp_4 \sim \SO_5$ and $\bigwedge^2\SL_4 \sim \SO_6$ by the exceptional isomorphisms. \\

In~\cite{Ho1} it is also proved that the restriction of an irreducible SMF module $(G,V)$ to a Levi subgroup $L$ of $G$ remains SMF. This fact can also be generalized to reducible modules.

\begin{lemma} \label{lemm:levi}
Let $V = \bigoplus_{i=1}^n V_i$ be a decomposition of a $G$ module into irreducible subspaces and denote by $P=LU$ a parabolic subgroup of $G$, where $L$ is a Levi subgroup and $U$ unipotent. If $V$ is SMF then for any choice of irreducible $L$ modules $W_i \subseteq {\rm Res}^G_L(V_i)$ the direct sum \[ W:= \bigoplus_{i=1}^n W_i, \] must also be SMF. 
\end{lemma}

\begin{proof} We slightly generalize the arguments of the original proof. Let $B$ the Borel subgroup of $G$ that is contained in $P$. Denote by $\left< W_i \right >_P$ the $P$-invariant subspace of $V_i$ generated by $W_i$, then we have a decomposition of vector spaces
\[ \left< W_i \right>_P = W_i \oplus W_i^+ \]
with $W_i^+$ being a $P$-invariant complement to $W_i$ in $\left< W_i \right>_P$. The subspace
\[ X:= \bigwedge^{p} \left( \bigoplus_{i=1}^n W_i^+ \right) \] 
with $p = \dim(W_1^+) + \dots + \dim(W_n^+)$ is one-dimensional. Hence, every element $b \in B \subseteq P$ acts as multiplication by $\lambda(b) := \det(b|_X)$ on $X$.

Choose any non-zero vector $t \in X$. Now take a highest weight vector $\nu \in \left( \bigwedge^r W \right)$ for $L$ of weight $\varphi$, i.e. $l(\nu) = \varphi(l) \cdot \nu$ for $l \in B_L$. We claim that $\nu \wedge t \in \bigwedge^{r + p} V$ is a highest weight vector for $G$ of weight $\lambda + \varphi$: 

Let $b\in B \subseteq P$ with $b = ul$ for $u \in U$, $l\in L$. Since $U \subseteq B$, we have $l = u^{-1}\cdot b \in B$ and so $l \in B_L = B \cap L$. Furthermore, $u(\nu) = \nu + \mu$ where $\mu$ is some element in $\bigwedge (W_1^+ \oplus\dots\oplus W_n^+)$. Thus, 
%
%
%
\begin{eqnarray*}
b(\nu\wedge t) &=& \varphi(l) \cdot u(\nu) \wedge b (t) \\
 &=& \lambda(b)\cdot \varphi(b) \cdot u(\nu) \wedge t \\
 &=& (\lambda+\varphi)(b) \cdot \nu\wedge t.
\end{eqnarray*}
This means, if the dimension of $(\bigwedge^r W)^{B_L,\varphi}$ is greater than one, the same is true for $(\bigwedge^{r + p} V)^{B,\lambda+\varphi}$. \end{proof}

\begin{corollary} \label{cor:products}
Let $G$ be a reductive group with simple factors $G_i$ together with an irreducible SMF representation $V = V_1 \otimes \dots \otimes V_k$. Then, $(G_i, V_i)$ is an irreducible SMF representation for all $1 \leq i \leq k$.
\end{corollary}

\begin{proof}
Up to a torus, every $G_i$ is a Levi component of $G$ and the restriction of $V$ to $G_i$ contains a copy of $V_i$. \end{proof}

A further application of the above Lemma is the following: For a classical group $G$ of rank $n$ every irreducible representation $V^{(n)}_{\lambda}$ of $G$ can be indexed by a certain partition $\lambda$ (see e.g. \cite{FH}). For $n \geq 3$ there is always a parabolic subgroup of $G$ such that the semisimple part of its Levi component is isomorphic to a simple group $G'$ of rank $n-1$ and of the same type as $G$. 
If $n$ is sufficiently large (compared to the number $\ell(\lambda)$ of boxes in the first column of $\lambda$), there is always $V^{(n-1)}_{\lambda} \subseteq V^{(n)}_{\lambda}$ as a $G'$ submodule. So once we have checked that $V^{(n)}_{\lambda}$ is not SMF, the same holds for all $V^{(m)}_{\lambda}$ with $m > n$. For example, this applies to the representations $V^{(n)}_{(k)} = S^k\SL_n$. \bigskip

\section{Centralizing actions} \label{sec:dual}

Given a reductive group $G$ acting on a vector space $W$, one can consider its centralizer algebra inside $\End(W)$ consisting of those endomorphisms that commute with all operators of $G$. This centralizer algebra plays a crucial role in our considerations since the multiplicities of the irreducible subrepresentations of $G$ in $W$ are given by the dimensions of the irreducible subrepresentations of this centralizer algebra. More precisely, one has the following

\begin{theorem}[Double commutant theorem] \label{thm:double-commutant}
Let $\mathcal{A} \subseteq \End(W)$ be a semisimple algebra and let $\mathcal{B} = \mathcal{A}' = \{ b \in \End(V): ba = ab ~for~all~ a\in\mathcal{A} \}$. Then, $\mathcal{B}$ is semisimple and $\mathcal{B}' = \mathcal{A}$. Moreover, as a joint $\mathcal{A}\otimes \mathcal{B}$ module one has

\[ W = \bigoplus_i V_i \otimes U_i, \]
where the $V_i$ (resp. $U_i$) are pairwise non-isomorphic representations of $\mathcal{A}$ (resp. $\mathcal{B}$).
\end{theorem}

\begin{proof} It follows from the semisimplicity of $\mathcal{A}$ and from Wedderburn theory that

\[ \mathcal{A} \simeq \bigoplus_i \End(V_i). \]

Now, the key idea is to show that $\mathcal{B} \simeq \bigoplus_i \End(U_i)$. For a complete proof see e.g. ~\cite[Theorem 3.3.7, p. 137]{GW}. 
\end{proof}

In this situation, we will call $(\mathcal{A},\mathcal{B})$ a {\it (reductive) dual pair} on $W$. Similarly, if $\mathcal{A}$
and/or $\mathcal{B}$ are generated by, e.g. group actions, we say that those groups form a dual pair on $W$. \bigskip

We will use Theorem~\ref{thm:double-commutant} to derive a characterization of SMF modules in terms of Clifford algebras. This will be useful for the treatment of reducible modules in Section~\ref{sec:reducible}. Let $Cl(V \oplus V^*)$ denote the Clifford algebra on a vector space $V \oplus V^*$ equipped with the canonical symmetric pairing $\left< \cdot,\cdot \right>$ given by

\begin{equation} \label{eqn:symmetric-pairing}
\left< (v,\lambda), (w, \mu) \right> := \lambda(w) + \mu(v),~~ v,w \in V,~ \lambda,\mu \in V^*.
\end{equation}

An action of a group $G$ on $V$ extends naturally to an action on $Cl(V \oplus V^*)$ and hence we can speak about the invariants $Cl(V \oplus V^*)^G$. 

\begin{prop} \label{prop:clifford}
$(G,V)$ is SMF if and only if $Cl(V\oplus V^*)^G$ is commutative.
\end{prop}

\begin{proof} $V$ is a maximal isotropic subspace of $V \oplus V^*$ with respect to $\left<\cdot,\cdot \right>$. Hence, the spin representation of $\SO(V\oplus V^*)$ gives rise to a $G$-equivariant algebra isomorphism $Cl(V \oplus V^*) \simeq \End(\bigwedge V)$. Taking invariants, it follows that 

\[ \mathcal{B} =\End_G(\bigwedge V) \simeq Cl(V \oplus V^*)^G \]
is the full centralizer algebra of the action $\rho: G \rightarrow \GL(\bigwedge V)$. Since this representation is completely reducible, $\rho(G)$ generates a semisimple subalgebra $\mathcal{A}$ of $\End(\bigwedge V)$. Hence, by the Double Commutant Theorem~\ref{thm:double-commutant}, 
\[ \bigwedge V = \bigoplus_i V_i \otimes U_i \]
is the isotypic decomposition as a joint $G \times Cl(V \oplus V^*)^G$ module. It follows that the SMF property of $(G,V)$ is equivalent to $\dim(U_i) = 1$ for all $i$. This in turn, is equivalent to the commutativity of $\mathcal{B} \simeq \bigoplus \End(U_i)$. 
\end{proof}

In the following, we present two particular dual pairs on exterior algebras. The discussion mainly follows~\cite{Ho1}. First, consider the action of $\GL_n\times \GL_m$ on $V = \CC^n\otimes \CC^m$. There is a well-known decomposition formula for the exterior powers which is also referred to as ``$(\GL_m,\GL_n)$ skew duality''.

\begin{prop} \label{prop:glm_gln}
$(\GL_n,\GL_m)$ is a reductive dual pair on $\bigwedge V$. More precisely, for $0 \leq k \leq m\cdot n$ one has
\begin{equation} \label{eqn:glm_gln}
\bigwedge^k V = \bigoplus_{\begin{minipage}{34pt}\tiny $|\lambda|=k,\\ \ell(\lambda)\leq m,\\ \lambda_1\leq n$ \end{minipage}} V_{\lambda}^{(n)}\otimes V_{\lambda^t}^{(m)}. 
\end{equation}
\end{prop}

\begin{proof} See~\cite[Theorem 4.1.1]{Ho1}.
\end{proof}

Moreover, also the highest weight vectors in~\eqref{eqn:glm_gln} can easily be described. (This goes back to~\cite[Cor. 8.2]{Ko}.) Arrange the standard basis vectors $e_{i,j} = e_i \otimes e_j$ of $\CC^n \otimes \CC^m$ in a rectangular array. Given a partition $\lambda$ with $\ell(\lambda) \leq n$ and $\lambda_1 \leq m$, we take the wedge product

\begin{equation} \label{eqn:glnglm-highest-weight-vectors}
\omega_\lambda = e_{1,1} \wedge\dots\wedge e_{1,\lambda_1} \wedge e_{2,1} \wedge\dots\wedge e_{2,\lambda_2} \wedge\dots \wedge e_{n,\lambda_n}.
\end{equation}

It is easy to see that this gives a highest weight vector for $\GL_n\times \GL_m$ of highest weight $(\lambda,\lambda^t)$. In particular, $\omega_\lambda$ has exterior degree $|\lambda|$. \bigskip

Next, consider the action of $\OG(V)$ on $V^m \simeq V \otimes \CC^m$. For convenience and later purpose, we outline the construction of the centralizer action. It is known~\cite[Theorem 4.3.2]{Ho1} that the invariants $(\bigwedge V \otimes \CC^m)^{\OG(V)}$ on the exterior algebra are generated by elements of degree $2$ which come from the various copies of the scalar product on $V$. Explicitly, if we take a basis $e_1,\dots,e_n$ of $V$ such that the symmetric form is given by $(e_i,e_j) = \delta_{i,n + 1 - j}$ and denote the copies of these basis elements by $e_{i,a}$, then for each pair $1 \leq a < b \leq m$ we get a basic invariant 

\begin{equation} \label{eqn:basic-invariant}
\xi_{a,b} = \sum_{i=1}^n e_{i,a}\wedge e_{n+1-i,b}.
\end{equation}
 
In particular, if we have $2m$ copies of $V$, there is a $\OG(V)$ isomorphism $V^{2m} \simeq V\otimes (\CC^m \oplus \CC^{*m})$. If we equip $\CC^m \oplus \CC^{*m}$ with a symmetric pairing as in \eqref{eqn:symmetric-pairing} we have, analogous as in the proof of Proposition~\ref{prop:clifford}, $\OG(V)$ isomorphisms

\begin{equation} \label{eqn:clifford-iso}
\bigwedge V\otimes (\CC^m \oplus \CC^{*m}) \stackrel{\varphi_1}{\longrightarrow} Cl(V\otimes (\CC^m \oplus \CC^{*m})) \stackrel{\varphi_2}{\longrightarrow} \End\left(\bigwedge V\otimes \CC^m\right). 
\end{equation}

It is well-known that $\varphi_1$ maps the space $\bigwedge^2 V\otimes (\CC^m \oplus \CC^{*m})$ to a Lie subalgebra of $Cl(V\otimes (\CC^m \oplus \CC^{*m}))$ isomorphic to $\mathfrak{so}(V\otimes (\CC^m \oplus \CC^{*m}))$. This subspace contains all the basic invariants~\eqref{eqn:basic-invariant} for $\OG(V)$. Denote exterior multiplication with an element $v \in V\otimes\CC^m$ by $\bigwedge v$ and interior multiplication with $\lambda\in (V\otimes\CC^m)^*$ by $\rfloor \lambda$. Then, the composite isomorphism~\eqref{eqn:clifford-iso} maps the basic invariants to the following operators: If $1 \leq a < b \leq m$, the element $\xi_{a,b}$ is is sent to

\begin{equation} \label{eqn:raising}
e^+_{a,b} = \sum_{i=1}^n \bigwedge{e_{i,a}}\bigwedge{e_{n+1-i,b}},
\end{equation}
that rises the degree of each homogeneous element by $2$. Furthermore, the image of $\xi_{a+m,b+m}$ is given by 

\begin{equation} \label{eqn:lowering}
e^-_{a,b} = \sum_{i=1}^n \rfloor{e^*_{i,a}}\rfloor{e^*_{n+1-i,b}}
\end{equation}
that lowers the degree by $2$. Finally, for $\xi_{a,b+m}$ with $1 \leq a , b \leq m$, we also get operators

\begin{equation} \label{eqn:preserving}
d_{a,b} = \sum_{i=1}^n \bigwedge{e_{i,a}}\rfloor{e^*_{n+1-i,b}} - \delta_{a,b} \cdot \frac{n}{2}
\end{equation}
that preserve degrees. The linear subspaces of $\mathfrak{so}(V\otimes(\CC^m \oplus \CC^{* m}))$ generated by these operators are denoted by $\mathfrak{so}^{(2,0)}$, $\mathfrak{so}^{(0,2)}$ and $\mathfrak{so}^{(1,1)}$ respectively. By virtue of the commutation relations for~\eqref{eqn:raising} -~\eqref{eqn:preserving}, their direct sum forms a Lie algebra isomorphic to $\mathfrak{so}_{2m}$. This discussion amounts to the following Proposition (see also~\cite[Theorem 4.3.4.1]{Ho1}):

\begin{prop} \label{prop:duality_construction}
The commutant of $\OG(V)$ inside $\End(\bigwedge V \otimes \CC^m)$ is generated as an associative algebra by a Lie algebra $\mathfrak{so}_{2m}$. \hfill $\Box$
\end{prop}

By the Double commutant Theorem ~\ref{thm:double-commutant} it follows that $\bigwedge \CC^n \otimes \CC^m$ is a module under the joint $(\OG_n, \mathfrak{so}_{2m})$ action and the isotypic components of $\OG_n$ and $\mathfrak{so}_{2m}$ stand in bijective correspondence. However, note that this is {\it not} an SMF module according to our definition since the wedge powers $\bigwedge^k \CC^n\otimes\CC^m$ are not stable with respect to action of $\mathfrak{so}_{2m}$. \bigskip

For the following we also need a description of the highest weights for $\OG_n$ that appear in $\bigwedge \CC^n\otimes \CC^m$. By construction~\eqref{eqn:glnglm-highest-weight-vectors}, the joint $\GL_n\times \GL_m$ highest weight vector $\omega_\lambda$ corresponding to the partition $\lambda$ is also a highest weight vector for $\OG_n$. 

Denote by $[\lambda]$ the character of the irreducible $\OG_n$ representation generated by $\omega_\lambda$. As long as $\ell(\lambda) \leq \lfloor \frac{n}{2} \rfloor$, it is no problem to read off this character directly. Otherwise, this is not immediately possible and moreover, $[\lambda]$ may refer to a virtual representation.~\cite{Ki1} gives a procedure (``modification rules'') how this representation can be determined. Here, the following notation is used: If $[\lambda]$ is the character of any irreducible representation $V$ of $\OG_n$, then $[\lambda]^\sharp$ denotes the character of $V \otimes \det$. 

\begin{lemma}[\bf Modification rules for $\OG_n$] \label{lemm:mod_rules} $\quad$ 

\begin{itemize}

\item[(a)] Remove $h = 2\ell(\lambda) - n$ boxes (called a {\itshape strip of length $h$}) from $\lambda$ by starting in the lowest box of the first column. Then in each step, if there is a box to the right, go one box right; otherwise go one box upwards.

\item[(b)] If the resulting diagram does not correspond to a regular partition $\hat{\lambda}$, then $[\lambda]\equiv 0$. Otherwise, we have $[\lambda]\equiv (-1)^{c-1}[\hat{\lambda}]^\sharp$, where $c$ is the number of columns that are affected by the removed strip. 

\item[(c)] If $\ell(\hat{\lambda}) > \lfloor \frac{n}{2} \rfloor$, repeat (a), otherwise $\hat{\lambda}=\tau$. \\

\end{itemize}

\end{lemma}

\begin{example}
Let $n = 6$ and $\lambda=(2^3,1)$. A strip removal of length $h=8-6=2$ leads to an irregular diagram, so that $[\lambda]=0$.

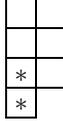
\[ \young(\hfill\hfill,\hfill\hfill,*\hfill,*) \]

More interesting is e.g. $\lambda=(2^5,1)$. The strip length equals $h = 12 - 6 = 6$. The modified diagram $\hat{\lambda}=(2,1^3)$ has length $\ell=4$, so a second removal is necessary, now $h' = 8 - 6 = 2$. This results in $[\lambda]\equiv -[2,1]$, because the first strip affected two columns of $\lambda$. In the picture, the $*-$signs (resp. $+-$signs) indicate the removal of the first (resp. second) strip:

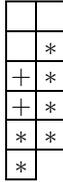
\[ \young(\hfill\hfill,\hfill*,+*,+*,**,*) \] \end{example}

Denote the conjugate partition of $\lambda$ by $\lambda^t = (\lambda^t_1,\dots,\lambda^t_m)$. According to \cite[Section 4.3.5, p. 57]{Ho1} the decomposition of $\bigwedge\CC^n\otimes\CC^m$ is determined by the {\it harmonic} highest weight
vectors $\omega_\lambda$, i.e. those in the kernel of $\mathfrak{so}^{(0,2)}_{2m}$, and this condition is equivalent to
$\lambda^t_i + \lambda^t_j \leq n$ for all $1 \leq i, j \leq m$. By applying the operators $d_{a,a}$ to a harmonic $\omega_\lambda$, one finds that this has highest weight

\begin{equation} \label{eqn:highest-weight_so2m}
\left(\frac{n}{2} - \lambda^t_m \right) \varepsilon_1 + \left(\frac{n}{2} - \lambda^t_{m-1}\right) \varepsilon_2 +\dots + \left(\frac{n}{2} - \lambda^t_1 \right) \varepsilon_m, 
\end{equation}
so it generates an irreducible representation $V(\tilde{\lambda})$ for $\mathfrak{so}_{2m}$. Now we are in the position to give the explicit decomposition of the dual pair $(\OG_n,\mathfrak{so}_{2m})$.

\newpage

\begin{prop}[$(\OG_n,\mathfrak{so}_{2m})$ skew duality] \label{prop:duality_explicit}
As a module for the joint $(\OG_n,\mathfrak{so}_{2m})$ action, we have 

\[ \bigwedge \CC^n\otimes \CC^m = \bigoplus_{\lambda:~ \lambda_1 \leq m,~\lambda^t_i + \lambda^t_j \leq n} V_{[\lambda]} \otimes V(\tilde{\lambda}). \]

Here, $V_{[\lambda]}$ is the irreducible representation of $\OG_n$ whose character equals $[\lambda]$. If $\ell(\lambda) > \lfloor \frac{n}{2} \rfloor$, it can be determined by the modification rule in Lemma~\ref{lemm:mod_rules}. $V(\tilde{\lambda})$ is the irreducible representation of $\mathfrak{so}_{2m}$ of highest weight~\eqref{eqn:highest-weight_so2m}.

Furthermore, the harmonic highest weight vectors $\omega_\lambda$ that correspond to this decomposition are given by the procedure~\eqref{eqn:glnglm-highest-weight-vectors}. In particular, $\omega_\lambda$ has degree $|\lambda|$. 
\hfill $\Box$
\end{prop}

The pairs $(\GL_n,\GL_m)$ and $(\OG_n,\mathfrak{so}_{2m})$ on $\bigwedge \CC^n\otimes\CC^m$ are special examples of reductive dual pairs and can be defined in a more general context. A detailed discussion on this topic can be found in~\cite{Ho2}.

\section{Classification of irreducible SMF modules} \label{sec:irred}

It turns out that every irreducible SMF module in this section can be traced back to one of the above dual pairs. In particular, $(\GL_n,\GL_m)$ skew duality~\eqref{eqn:glm_gln} itself gives  rise to a SMF representation. In our framework, one can say that the action $\CC^* \otimes \SL_n \otimes \SL_m$ (which is geometrically equivalent to $\GL_n \otimes \GL_m$) is SMF. Equivalently, under the semisimple subgroup $\SL_m \times \SL_n$, each submodule $\bigwedge^k \CC^n\otimes\CC^m$ is multiplicity-free. \bigskip

In general, let $G = \CC^* \times [G,G]$ with $[G,G]$ being semisimple and connected. We claim that there are only three further series of irreducible SMF modules.

\begin{theorem} \label{thm:SMF_irr}
Let $G$ be as above. A complete list (up to geometric equivalence) of irreducible SMF representations is given as follows: If $[G,G]$ is simple, it is given by one of the representations from Theorem~\ref{thm:howeslist}. Otherwise, it is one of the following: \bigskip

\noindent $\SL_m\otimes \SL_n~{\rm (} m,n\geq 2{\rm )}$, $\SL_2\otimes \SO_n~ {\rm (}n\geq 3 {\rm )}$, $\SL_3\otimes \SO_{2m+1}~ {\rm (}m\geq 1{\rm )},~ \SL_n \otimes \Sp_4~ {\rm (}n\geq 2{\rm )}$.
\end{theorem}

By Corollary \ref{cor:products}, all irreducible SMF representations for non-simple $[G,G]$ are obtained as products of representations in Theorem \ref{thm:howeslist}. In Appendix A we show by a case-by-case analysis that there are no other skew multiplicity-free representations than those in Theorem \ref{thm:SMF_irr}.

It then remains to show that each choice of $[G,G]$ from this list yields an SMF representation. As we pointed out, this is the case for $\SL_m\otimes \SL_n$. Note that $\SL_2\otimes\SO_n$ is SMF for all $n$ while for $\SL_3\otimes\SO_n$ this is only true for $n$ odd. Instead, if we consider the full orthogonal group, we can drop this assumption. We discuss this phenomenon in the subsequent lemmas. 

\begin{lemma} \label{lemm:SL3_Om}
$\SL_3 \otimes \OG_n$ is SMF for all $n \in \NN$.
\end{lemma}

\begin{proof} 
The dual pair $(\mathfrak{so}_{2k},\OG_n)$ on $\bigwedge \CC^k \otimes \CC^n$ as constructed in Proposition~\ref{prop:duality_construction} is closely connected to the action of $\SL_k\times \OG_n$ on the exterior algebra of $\CC^k \otimes \CC^n$: The key observation is that the subalgebra $\mathfrak{so}_{2k}^{(1,1)}$ of $\mathfrak{so}_{2k}$ is isomorphic to $\mathfrak{gl}_k$ and acts by degree preserving operators. In particular, we get an action of $\mathfrak{so}^{(1,1)}_{2k}\times \OG_n$ on $\CC^k\otimes \CC^n$. If one integrates this to a group action, then this becomes geometrically equivalent to $\CC^*\otimes\SL_k \otimes \OG_n$. 

In the current situation, $k = 3$ and by the bijective correspondence between the isotypic components of $\OG_n$ and $\mathfrak{so}_{2k}$, it suffices to show that the restriction $\mathfrak{so}_{6} \downarrow \mathfrak{sl}_3$ is multiplicity-free for irreducible representations of $\mathfrak{so}_{6}$. Note that it is known (see e.g.~\cite[Theorem 8.1.1, p. 350]{GW}) that the restriction of any irreducible representation of $\mathfrak{gl}_m$ to $\mathfrak{gl}_{m-1}$ is multiplicity-free. The claim follows since there is an exceptional isomorphism $\mathfrak{so}_6 \simeq \mathfrak{sl}_4$ and the embeddings of $\mathfrak{sl}_3$ into $\mathfrak{sl}_4$ and $\mathfrak{so}^{(1,1)}_6 \subseteq \mathfrak{so}_6$ give rise to a commutative diagram. 
\end{proof}

\begin{lemma} \label{lemm:SL3_SO2n+1}
$\SL_3 \otimes \SO_{n}$ is SMF if and only if $n$ is odd.
\end{lemma}

\begin{proof} Given the fact that $\SL_3 \otimes \OG_n$ is SMF, there is one problem that could cause a violation of the SMF property of $\SL_3 \otimes \SO_n$. The group $\SO_n$ acts trivially on the one-dimensional determinant representation $\CC_{\det}$ of $\OG_n$ and hence one has a $\SO_n$ equivariant isomorphism

\begin{equation}
V \otimes \CC_{\det} \simeq V
\end{equation}

for any representation $V$ of $\OG_n$. In terms of characters, this means that $[\lambda]^\sharp = [\lambda]$ as characters for the subgroup $\SO_n$. Concerning $(\OG_n,\mathfrak{so}_{6})$ skew duality, there are isotypic components $V_{[\lambda]}\otimes V(\tilde{\lambda})$ and $V_{[\tau]}\otimes V(\tilde{\tau})$ of $\bigwedge \CC^n\otimes \CC^3$ with 

\begin{equation} \label{eqn:mod_this}
[\tau] = [\lambda]^\sharp.
\end{equation}

Now, if ${\rm Res}_{\SL_3}(V(\tilde{\lambda}))$ and ${\rm Res}_{\SL_3}(V(\tilde{\tau}))$ have an irreducible $\SL_3$ module in common, $\SL_3\otimes\SO_n$ is no longer SMF, provided that both occur in the same exterior power $\bigwedge^k\CC^n\otimes\CC^3$.

If $n = 2m + 1$, we know from Proposition \ref{prop:duality_explicit} that $\lambda^t_i + \lambda^t_j \leq 2m + 1$ (and an analogous condition for $\tau$). Hence, if $\ell(\lambda) = \lambda^t_1 > m$, then $\lambda^t_i \leq m$. Thus, we can conclude that $\ell(\lambda) \leq m$ and $\ell(\tau) > m$ (or vice versa) and we need only a single strip removal (Lemma~\ref{lemm:mod_rules}) from $\tau$ in order to establish~\eqref{eqn:mod_this}.\bigskip

Now we restrict from $\SO_n \times \mathfrak{so}_6$ to $\SO_n \times \mathfrak{sl}_3$. In order to show that two copies of an irreducible $\SO_n \times \mathfrak{sl}_3$ submodule never appear in the same degree, we look at the harmonic highest weight vectors $\omega_\lambda$ of the dual pair $(\mathfrak{so}_6,\OG_{2m+1})$. By the description \eqref{eqn:glnglm-highest-weight-vectors}, we know that $\omega_\lambda$ has exterior degree $|\lambda|$. The action of $\OG_{2m+1}$ preserves this degree while $\mathfrak{so}_6$ acts by differential operators of degree $-2$, $0$ and $2$. Hence, for a fixed $\omega_\lambda$ all the weight spaces of the corresponding isotypic component of $\mathfrak{so}_6$ are contained either in $\bigoplus_i \bigwedge^{2i+1} V$ or in $\bigoplus_i \bigwedge^{2i} V$. \bigskip

Since $[\tau] = [\lambda]^\sharp$, we have to compare $\deg(\omega_\lambda)$ with $\deg(\omega_\tau)$. By Proposition~\ref{prop:duality_explicit}, $\lambda$ and $\tau$ have at most $3$ columns.  Lemma~\ref{lemm:mod_rules} says that the first part of $\tau$ changes to 

\begin{equation} \label{eqn:modi}
\tau_1^t - (2\tau_1^t - 2m - 1) = 2m + 1 - \tau_1^t
\end{equation}
while $\tau^t_2$ and $\tau^t_3$ remain unchanged. Note that $2m + 1 - \tau_1^t$ is odd if $\tau_1^t$ is even (and vice versa). In particular, 

\begin{equation} \label{eqn:difference_of_degrees}
\deg(\omega_\lambda) - \deg(\omega_\tau)
\end{equation}
is an odd integer and hence we are done. \bigskip

In the case that $n = 2m$, the same argument shows that the parity of~\eqref{eqn:difference_of_degrees} is even. Although this gives evidence that $\SL_3\otimes \SO_{2n}$ is not SMF, it is not a proof for that. But since for $m \geq 2$,

\[ {\rm Res}^{\GL_6}_{\OG_6}(V(3,2^{m-1},1)) \]
contains $V_{[2^{m-1}]}$ and $V_{[2^{m-1}]^\sharp}$ which are isomorphic as representations of $\SO_{2m}$, the SMF property of $\SL_3\otimes \SO_{2n}$ is violated. (For a detailed reasoning see the beginning of Appendix A).
\end{proof}

Let $X_n = \OG_n$ or $\SO_n$. 
From $(\OG_n,\mathfrak{so}_{4})$ skew duality we get a very precise description of the $\SL_2\times X_n$ module structure of $\bigwedge \CC^2\otimes\CC^n$. It also yields a refined statement of the SMF property of $\SL_2\otimes\OG_n$ which will be important in the next section. There is an isomorphism $\mathfrak{so}_4 \simeq \mathfrak{sl}_2 \times \mathfrak{sl}_2$ and we can make this explicit in terms of the operators~\eqref{eqn:raising} -~\eqref{eqn:preserving}. Namely, 

\[ [e^+_{1,2},e^-_{1,2}] = - d_{1,1} - d_{2,2} \]

\[ [d_{1,2}, d_{2,1}] = d_{1,1} - d_{2,2}. \]

The first equation stands for a copy $\mathfrak{sl}^{(1)}_2$ that raises and lowers the exterior powers and the second equation defines a copy $\mathfrak{sl}^{(2)}_2$ that preserves all degrees. \bigskip

In the current situation, the labels of the irreducible $\mathfrak{so}_4$ representations in Proposition~\ref{prop:duality_explicit} have the form $\tilde{\lambda} = (\frac{n}{2}-\lambda^t_2,\frac{n}{2}-\lambda^t_1)$. Since $V(a,b) \simeq V(a+b)\otimes V(a-b)$ via the isomorphism $\mathfrak{so}_4 \simeq \mathfrak{sl}^{(1)}_2 \times \mathfrak{sl}^{(2)}_2$, one deduces that

\begin{equation} \label{eqn:On_so4-skew}
\bigwedge \CC^n \otimes \CC^2 = \bigoplus_{\lambda_1 \leq 2,~ \lambda^t_1 + \lambda^t_2 \leq n} V_{[\lambda]}\otimes V(n - \lambda^t_1 - \lambda^t_2) \otimes V(\lambda^t_1 - \lambda^t_2) 
\end{equation}
as $X_n\times\mathfrak{sl}^{(1)}_2\times\mathfrak{sl}^{(2)}_2$ module. 
The isotypic decomposition of $\bigwedge \CC^n\otimes\CC^2$ under the group $\SL_2\times X_n$ can be determined as follows: For each $\lambda$ in the direct sum~\eqref{eqn:On_so4-skew}, there is a a harmonic highest weight vector $\omega_\lambda$ of degree $|\lambda| = \lambda^t_1 + \lambda^t_2$. Correspondingly, for each $\nu = 0, 1, \dots, n - \lambda^t_1 - \lambda^t_2 + 1$ we get a direct summand $V_{[\lambda]}\otimes V(\lambda^t_1 - \lambda^t_2)$ for the group $X_n\times\SL_2$ in 

\[ \bigwedge^{|\lambda| + 2\nu} \CC^n\otimes \CC^2. \]

The next lemma follows immediately from this description.

\begin{lemma} \label{lemm:Olmost_dual}
$\SL_2\otimes\OG_n$ is an almost dual pair on each exterior power of $\CC^n\otimes\CC^2$, i.e. if

\[ \bigwedge^k \CC^n\otimes\CC^2 = \bigoplus_\lambda V_{[\lambda]}\otimes V(\tilde{\lambda}), \]
then for all $\lambda \neq \mu$ in this sum, $V_{[\lambda]}$ and $V_{[\mu]}$ are nonisomorphic. In particular, $\SL_2\otimes\OG_n$ is SMF. \hfill $\Box$
\end{lemma}

\begin{lemma} \label{lemm:SOlmost_dual}
$\SL_2\otimes\SO_n$ is SMF for all $n \in \NN$. If $n$ is odd, then it is also an almost dual pair on each $\bigwedge^k\CC^2\otimes\CC^n$.
\end{lemma}

\begin{proof} Assume there are $\omega_\lambda$ and $\omega_\mu$ such that $[\lambda] = [\mu]^\sharp$ and the representations $V_{[\lambda]}\otimes V(\lambda^t_1-\lambda^t_2)$ and $V_{[\mu]}\otimes V(\mu^t_1-\mu^t_2)$ appear in the same exterior power of $\CC^n\otimes\CC^2$. We can assume that $\ell(\lambda) \leq \lfloor\frac{n}{2}\rfloor$ and $\ell(\mu) > \lfloor\frac{n}{2}\rfloor$. Since $\mu^t_1 + \mu^t_2 \leq n$, it follows from the modification rules in Lemma~\ref{lemm:mod_rules} that $\lambda_1 \neq \mu_1$ but $\lambda_2 = \mu_2$. Hence, $V(\lambda^t_1-\lambda^t_2)$ and $V(\mu^t_1-\mu^t_2)$ are nonisomorphic. Note that if $n = 2m$ and $\lambda = (\lambda_1,\dots,\lambda_m)$ with $\lambda \neq 0$, as a representation of $\SO_{2m}$

\[ V_{[\lambda]} = V_{[\lambda]}^- \oplus V_{[\lambda]}^+ \]
is a direct sum of irreducible representations of highest weights $(\lambda_1,\dots,\lambda_{m-1},-\lambda_m)$ and $\lambda$. But this does not affect the multiplicity-freeness.

If $n = 2m+1$, the above assumptions on $\omega_\lambda$ and $\omega_\mu$ are never satisfied (see proof of Lemma~\ref{lemm:SL3_SO2n+1}). Hence, the statement for $\SL_2\otimes\SO_{2m+1}$ follows.
\end{proof}

\begin{remark} Some parts of the above discussions can be found in~\cite[p. 147]{Ho1}. In fact, our deduction of the SMF property of $\SL_2\otimes\OG_n$ via $(\OG_n,\mathfrak{so}_4)$ skew duality follows the lines of this reference. On the other hand,~\cite{Ho1} does not cover the case $\SL_2\otimes\SO_n$. 
\end{remark}

We now turn to $\SL_n\otimes \Sp_4$. The exterior powers of this actions can be decomposed by virtue of $(\GL_n,\GL_4)$ skew duality~\eqref{eqn:glm_gln}. Namely,

\begin{equation} \label{eq:branching_dec}
 \bigwedge^k (\SL_n \otimes \Sp_4) \simeq \bigoplus_{\begin{minipage}{34pt}\tiny $|\lambda|=k,\\ \ell(\lambda)\leq n,\\ \lambda_1\leq m$ \end{minipage}} V_{\lambda}^{(n)}\otimes {\rm Res}^{\GL_4}_{\Sp_4}(V_{\lambda^t}^{(4)}), 
\end{equation}
and the multiplicity-freeness of this action is equivalent to the multiplicity-freeness of the restrictions ${\rm Res}^{\GL_4}_{\Sp_4}(V_{\lambda^t}^{(4)})$. In order to show this, one can once again impose exceptional isomorphisms of $\Spin$ groups. 

\begin{lemma} \label{lemm:branch_Sp4}
$\SL_n \otimes \Sp_4$ is SMF. 
\end{lemma}

\begin{proof}
The action of $\SL_4$ on $\bigwedge^2\CC^4$ identifies $\SL_4$ with $\Spin_6$. Since $\Sp_4$ fixes a symplectic form $\omega$, the restriction to this subgroup gives an irreducible representation on $\SO_5$. This embeds $\Spin_5$ into $\Spin_6$ as the subgroup $\{ g\in\Spin_6: g\omega = \omega \}$. Hence, the multiplicities for $\SL_4\downarrow \Sp_4$ are the same as those for $\Spin_6\downarrow\Spin_5$. It is well-known that the restriction $\Spin_{m+1}\downarrow\Spin_m$ is multiplicity-free for all $m$. ~\cite[p. 267ff]{Boe}. (See also [GW, p. 351ff].)
\end{proof}

\begin{remark}
The exposition in this section considerably profited from the referees' valuable remarks. Initially, we proved the SMF property of $\SL_3\otimes\SO_{2n+1}$ combinatorially by means of branching rules due to Littlewood~\cite{Li} and Kings modification rules. Also for $\SL_n\otimes\Sp_4$ we proved the multiplicity-freeness of branching rules explicitly using a description of~\cite[Theorem 12.1]{Su}. We thank the referee for pointing out the possible applications of $(\OG_n,\mathfrak{so}_{2m})$ skew duality in connection with the exceptional isomorphisms.
\end{remark}

\section{Reducible SMF modules} \label{sec:reducible}

While all irreducible SMF modules can simply be found by building products of simple SMF modules, the case of reducible representations is more involved. First, we have to eliminate trivial examples. If $(G_i,V_i)$ is SMF for $i=1,2$, then $(G_1\times G_2, V_1 \oplus V_2)$ is also SMF due to the isomorphism $\bigwedge (V_1 \oplus V_2) \simeq \bigwedge(V_1) \otimes \bigwedge (V_2)$.

\begin{definition}
A representation $(G,V)$ is called {\itshape indecomposable}, if it is not geometrically equivalent to $(G_1\times G_2, V_1\oplus V_2)$, with $V_i\neq 0$ and $G_i$ acting trivially on $V_j$ for $i\neq j$.
\end{definition}

In the following we are therefore only concerned with classifying indecomposable modules. There remains a difficulty: If in the above example, the $G_i$ are simple and one has an additional torus $\CC^*$ acting on $V_i$ as multiplication by $t^{a_i}$ with $a_i\neq 0$, this yields an indecomposable representation. But one can easily deduce that if we add a further torus this is again geometrically equivalent to a decomposable representation via an automorphism of $\CC^* \times \CC^*$.
This is one of the main reasons, why we consider saturated (indecomposable) representations only. For a discussion of the non-saturated case see Section~\ref{sec:non_saturated}. 

Recall that the exterior algebra of $V = V_1\oplus\dots\oplus V_l$ has a $l$-fold grading given by the subspaces 
\[ \bigwedge^{(r_1,\dots,r_l)} V = \bigwedge^{r_1} V_1 \otimes\dots\otimes \bigwedge^{r_l} V_l, \]
where $r_1 + \dots + r_l \leq \dim V$ and that a saturated representation of a reductive group $G$ is SMF whenever all the $\bigwedge^{(r_1,\dots,r_l)} V$ 
are multiplicity-free as modules for the semisimple part of $G$ (see also Remark~\ref{rem:saturated}).

A difficulty comes from the incompatibility of geometric equivalence with direct sums, e.g. 
$(\SL_n, \CC^n\oplus \CC^n) \sim (\SL_n,\CC^n\oplus\CC^{n*})$ holds only for $n = 2$. Fortunately, this can be disregarded due to the following fact.

\begin{lemma} \label{lemm:dual}
A saturated representation $(G, V_1\oplus V_2)$ is SMF if and only if $(G, V_1\oplus V_2^*)$ is SMF.
\end{lemma}

\begin{proof} Let us denote by $V$ either $V_1\oplus V_2$ or $V=V_1 \oplus V_2^*$. Obviously, the structure of $Cl(V\oplus V^*) \simeq \End(\bigwedge V)$ does not depend on this choice. Taking invariants, the claim follows from Proposition~\ref{prop:clifford}.
\end{proof}

The following definition is useful to describe the indecomposable representations in a systematic way.

\begin{definition}
Let $G$ be a group with $k$ simple factors and $V = V_1 \oplus \dots \oplus V_l$ the decomposition into irreducible components for $G$. The {\itshape representation diagram $\mathcal{G}$} of $(G,V)$ is a bipartite graph, given as follows:

Put $k$ vertices, each for one simple factor in $G$, in a horizontal line and put $\ell$ vertices, each for one irreducible submodule of $V$, in a line below. The top (resp. bottom) vertices are labeled by the simple factors (resp. irreducible summands). A pair of a top and bottom vertex is connected if and only if $G_i$ acts nontrivially on $V_j$. 
\end{definition}

\begin{remark}
An easy consequence of the above definitions is the fact that the connected representation diagrams are in $1-1$ correspondence with indecomposable representations and can thus be identified. Therefore the classification problem is reduced to the determination of all connected diagrams that lead to pairwise non-equivalent representations.
\end{remark}

It is reasonable to interpret subgraphs of $\mathcal{G}$ in terms of the underlying representations $(G,V)$. For every connected pair $(G_i,V_j)$ we can delete the corresponding edge in $\mathcal{G}$. If it happens that $V_j$ is no longer connected to any other group factor, we omit $V_j$ and the resulting graph $\mathcal{G}'$ stands for a $G$-submodule $V'$ of $V$. If, on the contrary, $G_i$ has no connecting edges left, $\mathcal{G}'$ is some kind of a restriction to a subgroup.

\begin{example}
The representation diagram $\mathcal{G}$ of the irreducible $G_1\times G_2$-module $V_1\otimes V_2$ is given by
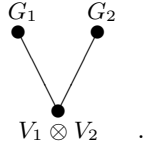
\[ \repV{G_1}{G_2}{V_1\otimes V_2}. \]
Let $\mathcal{G}'$ be the subgraph of $\mathcal{G}$ where the edge adjacent to $G_2$ and $V_1\otimes V_2$ is deleted. The $G_2$ vertex is no longer connected to any other vertex, so it can be omitted because we are only interested in indecomposable representations. So, $\mathcal{G}'$ is given by
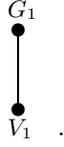
\[ \repI{G_1}{V_1}. \]

In this particular case, it follows from Lemma~\ref{lemm:levi} that if $\mathcal{G}$ is SMF, $\mathcal{G}'$ is necessarily SMF too. This is also true for arbitrary graphs:
\end{example}

\begin{prop} \label{prop:subgraph}
If $\mathcal{G}$ is the representation diagram corresponding to a SMF module, then also every representation belonging to a connected subgraph $\mathcal{G'}$ of $\mathcal{G}$ must be SMF.
\end{prop}

\begin{proof} It is an easy consequence of the above considerations that every subgraph of $\mathcal{G}$ can be obtained by a sequence of edge removals. We can therefore argue by induction, that the graph resulting from the removal of one edge still
corresponds to a SMF representation.

Let $G = G_1 \times\dots\times G_k$ and $V = V_1 \oplus\dots\oplus V_l$ with $V_i = W_i^1 \otimes\dots\otimes W_i^k$. Deleting one edge means, that for a unique simple factor $G_j$ and a unique submodule $V_i$ one has a ``reduced'' $G$-operation on $\tilde{V}_i = W_i^1 \otimes\dots\otimes W_i^{j-1}\otimes\CC\otimes W_i^{j+1} \otimes\dots\otimes W_i^k$.
So if we assume, by contradiction, that some subspace $\tilde{U} = \bigwedge^{r_1} V_1 \otimes\dots\otimes \bigwedge^{r_i} \tilde{V}_i \otimes\dots\otimes \bigwedge^{r_l} V_l$ is not multiplicity-free, we have to show that this is also the case for
\[ U:= \bigwedge^{r_1} V_1 \otimes\dots\otimes \bigwedge^{r_i} V_i \otimes\dots\otimes \bigwedge^{r_l} V_l. \]

There is a homomorphism from $G$ to the group $H:=\GL(V_1) \times\dots\times \GL(W_i^j) \times \GL(\tilde{V}_i) \times\dots\times \GL(V_l)$, so $U$ is a module for $H$. If $S_\lambda(\cdot)$ denotes the Schur functor corresponding to the partition $\lambda$, we get
\[ U \simeq \bigwedge^{r_1} V_1 \otimes\dots\otimes \left( \bigoplus_{|\lambda|=r_i} S_\lambda(W_i^j) \otimes S_{\lambda^t}(\tilde{V}_i) \right) \otimes\dots\otimes \bigwedge^{r_l} V_l \]
by $(\GL_n,\GL_m)$ skew duality \ref{prop:glm_gln}. In particular, $S^{r_i} W_i^j \otimes \tilde{U} \simeq \bigwedge^{r_1} V_1 \otimes\dots\otimes \left(S^{r_i} W_i^j \otimes \bigwedge^{r_i} \tilde{V}_i \right) \otimes\dots\otimes \bigwedge^{r_l} V_l$ is a $G$ submodule of $U$ which fails to be multiplicity-free. 
\end{proof}

Now, we can state the main result of this section.

\begin{theorem} \label{thm:SMF_red}
Let $V$ be a reducible indecomposable saturated representation of a reductive group $G$. We can assume that $G = (\CC^*)^k \times [G,G]$.  

\begin{itemize}

\item[a)] If $G$ is connected, then (up to geometric equivalence and a possible exchange of irreducible submodules by their duals) $([G,G],V)$ is given by one of the diagrams in the following list: \newline
$(\alpha)$ $\repA{\SL_n}{\CC^{n}}{\CC^n}$,                                        \hfill
$(\beta)$ $\repA{\SL_2}{\CC^2}{\!\!S^2\CC^2}$,                                    \hfill
$(\gamma)$ $\repA{\SL_3}{\CC^{3}}{S^2\CC^3}$,                                     \hfill
$(\delta)$ $\repA{\SL_2}{\CC^2}{S^3\CC^2}$,                                       \newline
$(\epsilon)$ $\repA{\SL_2}{\CC^2}{S^4\CC^2}$,                                     \hfill
$(\rho)$ $\repA{X_m}{\CC^{m}}{\CC^{m}}$,                             \hfill
$(\eta)$ $\repN{\SL_n}{\SL_m}{\CC^{n}}{\CC^n}{\CC^m}$\,\,,                        \newline
$(\kappa)$ $\repN{\SL_2}{\SL_n}{S^2\CC^2}{\CC^2}{\CC^n}$,                         \hfill
$(\vartheta)$ $\repN{\SL_2}{X_n}{\CC^2}{\!\!\CC^2}{\CC^n}$,   \hfill
$(\sigma)$ $\repN{\SL_2}{X_n}{\!\!\!S^2\CC^2}{\!\CC^2}{\CC^n}$, \newline
$\repW{\SL_n}{\SL_2}{\SL_m}{\CC^n}{\CC^2}{\CC^2}{\CC^m}$,                         \hfill 
$\repW{\SL_n}{\SL_2}{X_m}{\CC^n}{\CC^2}{\CC^2}{\CC^m}$,                            \hfill 
$\repW{X_{n}}{\SL_2}{X_{m}}{\CC^{n}}{\CC^2}{\CC^2}{\CC^{m}}$ \bigskip

\noindent Here, $X_n = \SO_n$ with $n$ be an odd integer. \bigskip

\item[b)] Every representation in a) with $X_n = \OG_n$ and $n$ be an arbitrary integer gives an instance of a SMF representation for a nonconnented group. 

\end{itemize}
\end{theorem}

\begin{proof}
For the modules $(\beta)-(\epsilon)$ the SMF property can be verified by explicit calculations. In the subsequent lemmas we treat all the remaining representations that are labeled by Greek letters and also the last diagram of type W. From this it follows that the other two diagrams consisting of three group factors are SMF.

In Appendix~\ref{app:reducible} it is shown that there are no other saturated indecomposable SMF modules than those on the above list. \end{proof}


\begin{lemma}
Representations $(\alpha)$ and $(\eta)$ of Theorem~\ref{thm:SMF_red} are SMF. 
\end{lemma}

\begin{proof} Consider the action of $\SL_n \times \SL_{m + 1}$ on $\CC^n\otimes \CC^{m+1}$. For any $0\leq k \leq n(m + 1)$ we have

\begin{equation}
\bigwedge^k(\CC^n \otimes \CC^{m+1}) = \bigoplus_{|\mu| = k,~ \ell(\mu) \leq n, ~\mu_1 \leq m + 1}V^{(n)}_{\mu}\otimes V^{(m+1)}_{\mu^t}
\end{equation}
From this we can obtain the desired decomposition if we consider the Levi subgroup of $\SL_{m+1}$ that corresponds to deleting the rightmost node in the Dynkin diagram $A_m$. Then, the interlacing condition for $\GL_{m+1} \downarrow \GL_m$~\cite[Theorem 8.1.1, p. 350]{GW} says that $\Gamma^{(m+1)}_{\mu}$ breaks up under $\SL_m$ into the direct sum of all representations $\Gamma^{(m)}_{\nu}$ such that 
\[ \mu_1 \geq \nu_1 \geq \mu_2 \geq \nu_2 \geq \dots \geq \nu_n \geq \mu_{n+1}. \]

Since they are all pairwise different, this concludes the proof for $(\eta)$. The SMF property for $(\alpha)$ now follows from Proposition~\ref{prop:subgraph}.
\end{proof}

\begin{lemma}
Representation $(\kappa)$ of Theorem~\ref{thm:SMF_red} are SMF. 
\end{lemma}

\begin{proof} Let $W = S^2\SL_2$. Note that $\bigwedge^i W$ is either isomorphic to $S^2\SL_2$ or to the trivial representation. This ensures all tensor products $\bigwedge^i W \otimes V_{\mu}$ to be multiplicity-free. In particular, all the subspaces

\[ \bigwedge^i(W) \otimes \bigwedge^j(\SL_2\otimes\SL_n) \simeq \bigoplus_{|\mu|=j,~\ell(\mu)\leq 2}\left(\bigwedge^i W \otimes V^{(2)}_{\mu}\right) \otimes V^{(n)}_{\mu^t} \] 
are multiplicity-free and hence the claim is proved.
\end{proof}

For the modules with a factor $X_n$, in all cases the proofs rely on the statement of Lemma~\ref{lemm:Olmost_dual} and~\ref{lemm:SOlmost_dual} that $\SL_2\otimes X_n$ is an almost dual pair on the various $\bigwedge^k\CC^2\otimes\CC^n$.

\begin{lemma}
$X_n \otimes\SL_2 \oplus \SL_2\otimes X_m$ (possibly with $X_n = \SL_n$ and/or $X_m = \SL_m$) is SMF. 
\end{lemma}

\begin{proof} Since $\SL_2\otimes X_n$ is an almost dual pair (even for $X_n = \SL_n$), it follows that for all the subspaces

\[ \bigwedge^{(k,l)}(\CC^n\otimes\CC^2 \oplus \CC^2\otimes\CC^m) = \bigoplus_{\lambda,\mu} V^{(n)}_{[\lambda]} \otimes V(\tilde{\lambda})\otimes V(\tilde{\mu}) \otimes V^{(m)}_{[\mu]}, \]
we have $V^{(n)}_{[\lambda]} \neq V^{(n)}_{[\lambda']}$ for $\lambda \neq \lambda'$. The claim is now immediate, since the tensor product $V(\tilde{\lambda})\otimes V(\tilde{\mu})$ of irreducible $\SL_2$ representations is multiplicity-free. 
\end{proof}

\begin{lemma}
Representations $(\sigma)$ and $(\vartheta)$ of Theorem~\ref{thm:SMF_red} are SMF. 
\end{lemma}

\begin{proof} We can argue as in the preceding lemmas, using that all exterior powers $\bigwedge^k W$ are irreducible. Hence,

\[ \bigwedge^{(k,l)}(W \oplus \CC^2\otimes\CC^n) = \bigoplus_{\lambda,\mu} V(\tilde{\lambda})\otimes V(\tilde{\mu}) \otimes V^{(m)}_{[\mu]} \]
is multiplicity-free for all $k$ and $l$. \end{proof}

\begin{lemma} 
Representation $(\rho)$ of Theorem~\ref{thm:SMF_red} is SMF. 
\end{lemma}

\begin{proof} Let $V = \CC^n \oplus \CC^n$. Consider the isomorphism $\bigwedge \CC^n\otimes\CC^2 \simeq \bigwedge (\CC^n \oplus \CC^n) \simeq \bigoplus_{k,l} \bigwedge^{(k,l)} V$ of $\OG_n$ modules. Recall the description~\eqref{eqn:On_so4-skew} of the $\OG_n\times\mathfrak{sl}^{(1)}_2\times\mathfrak{sl}^{(2)}_2$ module structure. Here, the root operators of $\mathfrak{sl}^{(1)}_2$ act as

\[  \bigwedge^{(k,l)} V \longrightarrow \bigwedge^{(k+\epsilon,l+\epsilon)} V \]
with $\epsilon = \pm 1$, whereas those of $\mathfrak{sl}^{(2)}_2$ act as

\[ \bigwedge^{(k,l)} V \longrightarrow \bigwedge^{(k+\epsilon,l-\epsilon)} V. \] 

Recall that $\SL_2\times X_n$ is an almost dual pair on $\bigwedge^p \CC^2\otimes\CC^n$. If we restrict to $X_n$, we get multiple copies of $V_{[\lambda]}$ in $\bigwedge^{|\lambda| + 2i} V$ for each $\omega_\lambda$, but they are contained in different $\bigwedge^{(k,l)} V$ with $k + l = |\lambda| + 2i$. Since we assume saturatedness, this representation is SMF.
\end{proof}

\section{Non-saturated Modules} \label{sec:non_saturated} 

The question of determining the SMF property of an arbitrary (not necessarily saturated) representation is quite difficult which comes, amongst others, from the absence of a simple geometric or algebraic characterization of skew multiplicity-free actions. 

If we take a finite set of saturated indecomposable representations $(G_i,V_i)$ from Theorems~\ref{thm:howeslist}, \ref{thm:SMF_irr} and~\ref{thm:SMF_red}, the product $G:= G_1 \times \dots \times G_n$, $V:=V_1 \oplus \dots \oplus V_n$ yields a saturated SMF module. If we choose a subgroup of the torus factors of $G$, it may or may not happen that $\bigwedge V$ stays multiplicity-free under restriction.

This description is quite unsatisfactory, because the last instruction may involve the complete computation of decomposition of exterior powers $\bigwedge^k V$. In the symmetric setting there is a more direct proposition (which can be found in~\cite{Lea}) concerning the character group $X=X(T)$ of the torus factors $T$ of $G$:

Subgroups $U$ of $T$ correspond to a quotient of $X$ by $X_U=\{ \chi \in X:~ \chi|_U=1 \}$. Let $S$ denote be the group generated by the highest weights occurring in $S(V)$, intersected by $X$. Under restriction to $U$, $S(V)$ stays multiplicity-free if and only if the quotient map $S \rightarrow X/X_U$ is injective.

The proof uses the fact that for a multiplicity-free module $S(V)$ the elements therein invariant under the unipotent radical of $G$ form a polynomial algebra. A similar statement for the skew symmetric setting is not known yet.

\begin{example}
Consider the action of $(\CC^*)^2\times \SL_3$ on $V = S^2\CC^3 \oplus \CC^3$ (which is labeled by $(\gamma)$ in Theorem~\ref{thm:SMF_red}). One can check that all subspaces $\bigwedge^r V$ are multiplicity-free for $\SL_3$. Hence if we take the pullback coming from the diagonal embedding of $\CC^*$ in $(\CC^*)^2$, the resulting action of $\CC^*\times\SL_3$ on $V$ is still SMF. But if we take instead the embedding 
\[ z \mapsto (z^2,z), \] 
we see that the torus acts by $z^6$ on the submodules $\bigwedge^3\CC^3$ and $\bigwedge^6 S^2\CC^3$ of $\bigwedge V$ which are both (as representations of the simple factor $\SL_3$) isomorphic to the trivial representation. Thus, in this case the action is no longer SMF. 
\end{example}

Similarly, we can see that the assumption in Lemma~\ref{lemm:dual} of a saturated operation is essential. If we take the action of $\CC^*\times\SL_3$ via the diagonal embedding but now acting on $V' = S^2\CC^3 \oplus \CC^{3*}$, the standard representation $\CC^3$ has multiplicity two inside $\bigwedge^2 V'$. Since $\CC^*$ acts by $z^2$ on the whole subspace, we conclude that $(\CC^*\otimes\SL_3, V)$ is SMF while $(\CC^* \otimes \SL_3, V')$ is not. \\

One can detect similar phenomena for the other modules in Theorem~\ref{thm:SMF_red}. For instance, $(\alpha)$ stays SMF when restricted to $\CC^*\times \SL_n$ (diagonal embedding). But here, we can also exchange $\CC^n \oplus \CC^n$ by $\CC^n \oplus \CC^{n*}$ without violating the SMF condition. \\

At least it is satisfactory that there is no other possibility of constructing SMF modules than the one described above.

\begin{lemma}
Every skew multiplicity-free representation $(G,V)$ is obtained by restrictions from a sum of indecomposable saturated SMF representations.
\end{lemma}

\begin{proof} By adding torus factors, we can assume that $(G,V)$ comes from a saturated representation $(\tilde{G},V)$. By definition, this is a sum of indecomposable representations $(G_i,V_i)$, which can be easily recognized as saturated ones.
\end{proof}

\appendix

\section{Exhaustion of irreducible SMF modules} \label{app:irred}

We now complete the proof of Theorem \ref{thm:SMF_irr}. Therefore, we show that all products of representations in Theorem \ref{thm:howeslist} which have not been verified yet, fail to be SMF. Let us start with an observartion on $(\GL_n,\GL_m)$ skew duality. If we have irreducible representations $G \rightarrow \GL(V)$ and $H \rightarrow \GL(W)$ of simple groups, we know that the exterior product $\bigwedge^k (V \otimes W)$ decomposes as $\GL(V)\times\GL(W)$ module by

\begin{equation} \label{eqn:plethysm-duality}
\bigoplus_{|\lambda|=k,~ \ell(\lambda)\leq \dim V,~ \lambda_1 \leq\dim W} S_{\lambda}(V) \otimes S_{\lambda^t}(W). 
\end{equation}
Hence, if the plethysms $S_{\lambda}(V)$ or $S_{\lambda^t}(W)$ contain multiplicities for any $\lambda$ in this direct sum, $V\otimes W$ is not a SMF module for $G\times H$. In Table~\ref{tab:NonMFplethysms} we give examples of such $V$ and $\lambda$. The labels refer to Theorem~\ref{thm:howeslist}, $\lambda$ is the partition of the corresponding plethysm $S_\lambda(V)$ and under ``Representation'' we give the multiplicity of an irreducible submodule and its highest weight with respect to the basis of fundamental weights. Note that in all those cases, $\ell(\lambda) = 2$, i.e. the term corresponding to that partition always occurs whenever $\dim V \geq 2$. But this condition is always satisfied and hence, none of these $W$ gives rise to a SMF representation $V\otimes W$. \bigskip

\begin{table}	
\[ \begin{array}{| l l l l |} \hline
{\rm Label}  & V                   & \lambda   & {\rm Representation} \\ \hline
(A2^3)       & S^2\SL_3            & (2^3)     & 2[2,2]                \\ 
(A3)         & S^3\SL_2            & (2^2)     & 2[4]                  \\
(Ak)         & S^k\SL_2,~k = 4,5,6 & (2,1)     & 2[k]                  \\
(A7)         & S^3\SL_3            & (2^2)     & 2[2,2]                \\
(A8^5)       & \bigwedge^2\SL_5    & (2^3,1)   & 2[2,0,1,1]            \\
(A9)         & \bigwedge^3\SL_6    & (2^2,1)   & 3[0,1,1,1,0]          \\
(B2)         & \Delta_7            & (2^2,1^2) & 3[0,1,0]              \\
(B3)         & \Delta_9            & (2^2,1)   & 3[0,0,0,1]            \\
(C1^3)       & \Sp_6               & (2^2,1^2) & 2[0,1,0]              \\
(C3)         & \bigwedge^3_0\Sp_6  & (2^2)     & 2[0,2,0]              \\
(D2)         & \Delta^+_{10}       & (2^4)     & 2[0,0,2,0,0]          \\
(D3)         & \Delta^+_{12}       & (2^2,1)   & 2[1,0,0,0,0,1]        \\
(E6)         & E_6                 & (2^3)     & 2[1,0,0,0,0,1]        \\
(E7)         & E_7                 & (2^2,1)   & 2[1,0,0,0,0,0,1]      \\
(G)          & G_2                 & (2^2)     & 2[2,0]                \\ \hline
\end{array} \]
	\caption{Non-multiplicity-free plethysms of irreducible SMF representations}
	\label{tab:NonMFplethysms}
\end{table}

Among the representations of Theorem~\ref{thm:howeslist}, only $(A1^n)$, $(B1^n)$, $(D1^n)$  do not occur in Table~\ref{tab:NonMFplethysms}. (Note that $(A2^2) = (B1^1)$, $(A8^4) = (D1^3)$ and $(C2) = (B1^2)$ by geometrical equivalence.) 
Hence, the label of any irreducible SMF representation $V\otimes W$ must be $(X1^n)(Y1^m)$ with $X$ and $Y$ either of the letters $A$, $B$, $C$ or $D$. And if one of them equals $C$, then its exponent must be $2$. \bigskip

In order to show that this does not yield other SMF representations than those in Theorem~\ref{thm:SMF_irr}, it suffices (by Lemma~\ref{lemm:levi} and another application of the plethysm argument~\eqref{eqn:plethysm-duality}) to check the modules

\begin{equation} \label{eqn:remaining}
\SL_4\otimes\SO_3,~ \SL_3\otimes\SO_6,~ \SO_3\otimes\SO_3,~ \Sp_4\otimes\Sp_4. 
\end{equation}

(Since, for instance, $\SL_3\otimes\SO_6$ is a Levi component of $\SL_4\otimes\SO_6$. Furthermore, by~\eqref{eqn:plethysm-duality} multiplicities for $\SL_4\otimes\SO_6$ also produce multiplicities for $\Sp_4\otimes\SO_6$.) \bigskip

Note that for the modules in~\eqref{eqn:remaining}, the plethysms in~\eqref{eqn:plethysm-duality} are actually restrictions of irreducible $\GL(V)$ modules to the corresponding subgroups. They can be easily calculated by well-known branching rules from Littlewood (and application of the modification rules~\ref{lemm:mod_rules}). They state that for any partition $\lambda$,

\[ \{\lambda\} = \sum_{\mu} \left(\sum_{\delta~even}\, c^{\lambda}_{\delta,\mu} \right) [\mu] \]
for the characters of $\OG_n$, while for $\Sp_{2n}$ they decompose as

\[ \{ \lambda \} = \sum_{\mu} \left( \sum_{\beta^t~\!even}\, c^{\lambda}_{\beta,\mu} \right) \left<\mu\right>. \]

Note that these are identities of ``universal'' characters (why one might impose the modification rules). For a rigorous definition of these objects and a proof of this identity see~\cite[Theorem 2.3.1]{KT}. 

In our case, we calculate some specific restrictions that contain multiplicities. First, we see that

\begin{equation}
{\rm Res}^{\GL_3}_{\OG_3}(S_{(4,2)}(V)) = V_{[4]} \oplus V_{[3]^\sharp} \oplus 2 \cdot V_{[2]} \oplus V_{[0]}
\end{equation}
and hence, $\SL_4 \times \SO_3$ is not SMF. Similarly, for $\SL_3 \times \SO_6$ note that

\[ {\rm Res}^{\GL_6}_{\OG_6}(S_{(3,2,1^2)}(V)) = V_{[2,1]} \oplus V_{[2,1]^\sharp} \]
and that $V_{[2,1]} \simeq V_{[2,1]^\sharp}$ as representations of $\SO_6$. Finally, as a module for the group $\SL_m\times\SL_m$, the fourth exterior power of $\bigwedge^4\CC^m\otimes\CC^m$ contains the irreducible submodules $V_{(3,1)}\otimes V_{(2,1^2)}$ and $V_{(2,1^2)}\otimes V_{(3,1)}$. Hence, for $\SO_3 \times \SO_3$ (if $m=3$) and $\Sp_4\times\Sp_4$ (if $m=4$) it contains multiplicities since

\[ {\rm Res}^{\GL_3}_{\OG_3}(S_{(3,1)}(V)) = V_{[3]^\sharp} \oplus V_{[2]} \oplus V_{[1]^\sharp} {\rm~and~} {\rm Res}^{\GL_3}_{\OG_3}(S_{(2,1^2)}(V)) = V_{[2]^\sharp} \oplus V_{[1]} \]
in the first case, while in the second case

\[ {\rm Res}^{\GL_4}_{\Sp_4}(S_{(3,1)}(V)) = V_{\left<3,1\right>} + V_{\left<2\right>} {\rm~and~} {\rm Res}^{\GL_4}_{\Sp_4}(S_{(2,1^2)}(V)) = V_{\left<2\right>} \oplus V_{\left<1^2\right>}. \]

By this we know that if $G$ is a group with two simple factors, there are no further irreducible SMF representations than those in Theorem~\ref{thm:SMF_irr}. In this list, there is one instance of an SMF representation for a group with three simple factors, since $\SL_2\otimes\SO_4 \sim \SL_2\otimes\SL_2\otimes\SL_2$. We claim that this is the only case of such group, i.e. that all other groups with more than two simple factors admit no irreducible SMF representation: If there is such a space, then at least two simple factors must be isomorphic to $\SL_n$. Therefore it suffices to show that the following actions are not SMF: 
\[ \begin{array}{| l c l |} \hline
{\rm Label}                                & \bigwedge^a & {\rm Representation} \\ \hline
\SL_2\otimes \SL_2 \otimes \SL_3           & 6           & 2[2][2][1,1]          \\
\SL_2\otimes\SL_2\otimes S^2 \SL_2         & 5           & 3[1][1][2]            \\
\SL_2\otimes \SL_2\otimes \bigwedge^2\SL_4 & 3           & 2[1][1][0,1,0]        \\
\SL_2\otimes\SL_2\otimes\Sp_4              & 5           & 3[1][1][1,0]          \\ \hline
\end{array} \]

In particular, this completes the proof of Theorem \ref{thm:SMF_irr}.

\section{Exhaustion of saturated SMF modules} \label{app:reducible}

We begin with two lemmas which make a number of calculations superfluous and are thus very helpful. Let $G$ be a group with an irreducible representation $V$. Given this, one can construct a representation of $G_1 = (\CC^*)^2 \times G$ on the direct sum $V\oplus V$ with the $\CC^*$ factors acting as in~\eqref{eqn:saturated}. On the other hand, we have a representation of $G_2 = \CC^* \times \SL_2 \times G$ on $\CC^2\otimes V$. The embedding of $G_1$ into $G_2$ with $\CC^*$ being identified as the maximal torus of $\SL_2$ is compatible with these actions. Hence, multiplicities in $\bigwedge^k \CC^2\otimes V$ with respect to $G_2$ would lead to multiplicities in $\bigwedge^{(p,q)} (V\oplus V)$ with respect to $G_1$ for some $p + q = k$. 

\begin{lemma} \label{lemm:B1}
If $(G_1,V\oplus V)$ is SMF, then $(G_2,\CC^2\otimes V)$ is also SMF. \hfill $\Box$
\end{lemma}

\begin{lemma} \label{lemm:B2}
Let $V$ be a representation of $G$ such that $\bigwedge^k V$ is reducible for some integer $k$. Then, $(\CC^* \times \CC^* \times G, V\oplus V)$ is not SMF.
\end{lemma}

\begin{proof}
Assume without loss of generality that $\bigwedge^k V = V_1 \oplus V_2$. Then, 

\[ \bigwedge^k V \otimes \bigwedge^k V = V_1\otimes V_1 \oplus V_2\otimes V_2 \oplus 2 \cdot (V_1\otimes V_2) \]

\noindent is not multiplicity-free for $G$. Hence, $\bigwedge (V \oplus V)$ is not multiplicity-free for $(\CC^*)^2 \times G$.
\end{proof}

The completeness of the list in Theorem~\ref{thm:SMF_red} is shown. We start with considering saturated indecomposable representations with two nontrivial submodules and identify them with their representation diagrams. The first three types being discussed are
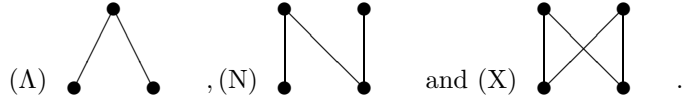
\[ 
 {\rm (\Lambda)}
\begin{picture}(60,40)  
 \put(25,30){\circle*{5}}
 \put(10,0){\circle*{5}} \put(40,0){\circle*{5}}  
 \put(10,0){\line(1,2){15}} \put(40,0){\line(-1,2){15}}
\end{picture}, {\rm (N)}
\begin{picture}(60,40)  
 \put(10,30){\circle*{5}} \put(40,30){\circle*{5}} 
 \put(10,0){\circle*{5}} \put(40,0){\circle*{5}}  
 \put(10,0){\line(0,1){30}} \put(40,0){\line(0,1){30}} \put(10,30){\line(1,-1){30}}
\end{picture}
{\rm ~and~(X)} 
\begin{picture}(60,40)  
 \put(10,30){\circle*{5}} \put(40,30){\circle*{5}} 
 \put(10,0){\circle*{5}} \put(40,0){\circle*{5}}  
 \put(10,0){\line(0,1){30}} \put(40,0){\line(0,1){30}} \put(10,30){\line(1,-1){30}} \put(10,0){\line(1,1){30}}
\end{picture}.
\]

In the first case, by Proposition~\ref{prop:subgraph}, a SMF representation is the direct sum of two irreducible SMF representations $V$ and $W$. Clearly, they both must satisfy the conditions of Lemma~\ref{lemm:B1} and~\ref{lemm:B2}. All necessary counter-examples (recall also Lemma~\ref{lemm:levi} concerning Levi factors) are listed in Table~\ref{tab:NonSMF_A}. 

\begin{table}
\[\begin{array}{| l l || l l |} \hline
\bigwedge^{(2,5)}\SL_4 \oplus S^2 \SL_4   & 2[1,2,1]     & \bigwedge^{(1,2)}S^3 \SL_2 \oplus S^4 \SL_2  & 2[3]  \\
\bigwedge^{(1,3)}\SL_2 \oplus S^5 \SL_2   & 2[4]         & \bigwedge^{(1,2)}S^3 \SL_2 \oplus S^5 \SL_2  & 2[3]  \\
\bigwedge^{(1,3)}\SL_2 \oplus S^6 \SL_2   & 2[5]         & \bigwedge^{(1,2)}S^3 \SL_2 \oplus S^6 \SL_2  & 2[3]  \\ 
\bigwedge^{(1,3)}\SL_3 \oplus S^3 \SL_3   & 2[3,2]       & \bigwedge^{(1,2)}S^4 \SL_2 \oplus S^5 \SL_2  & 3[4]   \\
\bigwedge^{(2,3)}\SL_4 \oplus \bigwedge^2 \SL_4 & 2[1,0,1]& \bigwedge^{(1,2)}S^4 \SL_2 \oplus S^6 \SL_2  & 2[2]   \\
\bigwedge^{(2,2)}\SL_6 \oplus \bigwedge^3 \SL_6 & 2[0,1,0,0,0] & \bigwedge^{(1,2)}S^5 \SL_2 \oplus S^6 \SL_2 & 2[3] \\
\bigwedge^{(1,2)}S^2\SL_2 \oplus S^3\SL_2 & 2[2] & \bigwedge^{(1,3)}\bigwedge^2\SL_4\oplus\bigwedge^2\SL_4 & 2[1,0,1]  \\ 
\bigwedge^{(1,2)}S^2\SL_2\oplus S^4\SL_2 & 2[4] & \bigwedge^{(1,2)}\bigwedge^2\SL_6 \oplus \bigwedge^3\SL_6 & 2[0,1,0,0,0] \\
\bigwedge^{(1,2)}S^2\SL_2\oplus S^5\SL_2        & 2[2]     & \bigwedge^{(2,2)}\Sp_4 \oplus\bigwedge^2_0 \Sp_4 & 2[2,0]   \\
\bigwedge^{(1,2)}S^2\SL_2\oplus S^6\SL_2        & 2[4]     & \bigwedge^{(2,2)}\Sp_6\oplus\bigwedge^3_0\Sp_6   & 2[0,1,0] \\
\bigwedge^{(2,2)}S^2\SL_4\oplus\bigwedge^2\SL_4 & 2[2,1,0] & \bigwedge^{(2,3)}\Spin_8 \oplus \Delta^+_8 & 2[1,0,1,0]  \\
\bigwedge^{(2,2)}S^2\SL_6\oplus\bigwedge^3\SL_6& 2[2,0,1,0,1] & &  \\ \hline 
\end{array}\]
 \caption{Non-SMF modules of type ($\Lambda$)}
	\label{tab:NonSMF_A}
\end{table}

SMF representations of the type (N) must contain a subgraph corresponding to $(\alpha) - (\rho)$ in Theorem~\ref{thm:SMF_red}. Furthermore, the submodule on which both simple factors of $G$ act nontrivially must be SMF. The
labels in Table~\ref{tab:NonSMF_N} indicate the sorting by $(\Lambda)$-subgraphs.

\begin{table}
\[\begin{array}{| l l || l l |} \hline
(\alpha)                                                    &            & (\beta) &              \\ \hline
\bigwedge^{(1,3)}\SL_3\oplus \SL_3\otimes S^2\SL_2&2[1,0][2]&\bigwedge^{(1,4)}S^2\SL_2\oplus \SL_2\otimes \bigwedge^2\SL_4   & 2[2][1,0,1]   \\
\bigwedge^{(1,4)}\SL_2 \oplus \SL_2\otimes \bigwedge^2\SL_4 & 2[3][1,0,1] & \bigwedge^{(1,3)}S^2\SL_2\oplus \SL_2\otimes \Sp_4    & 2[1][1,0]     \\
\bigwedge^{(1,3)}\SL_2\oplus \SL_2\otimes\Sp_4 & 2[2][1,0]& \bigwedge^{(1,3)}\SL_2 \oplus S^2\SL_2\otimes\SL_2  & 2[3][1]  \\
\hline
(\delta) &                                                        & (\epsilon)  &    \\ \hline
\bigwedge^{(2,2)}S^3\SL_2\oplus\SL_2\otimes\SL_2 & 2[2][0] & \bigwedge^{(2,2)}S^4\SL_2\oplus\SL_2\otimes \SL_2 & 2[4][0] \\
\bigwedge^{(2,2)}S^3\SL_2\oplus\SL_2\otimes S^2\SL_2 & 2[2][2] & \bigwedge^{(2,2)}S^4\SL_2\oplus\SL_2\otimes S^2\SL_2 & 2[4][2] \\   
\bigwedge^{(2,2)}S^3\SL_2 \oplus \SL_2\otimes \bigwedge^2 \SL_4   & 2[2][1,0,1] & \bigwedge^{(2,2)}S^4\SL_2 \oplus \SL_2\otimes \bigwedge^2 \SL_4  & 2[4][1,0,1]  \\ 
\bigwedge^{(1,3)}S^3\SL_2\oplus\SL_2\otimes\Sp_4 &2[2][1,0] & \bigwedge^{(1,3)}S^4\SL_2\oplus\SL_2\otimes\Sp_4 & 2[3][1,0] \\  
\hline
(\gamma)  &                                            & (\rho)   &  \\ \hline
\bigwedge^{(2,3)}S^2\SL_3\oplus \SL_3\otimes\SL_2      & 2[2,1][1]  & \bigwedge^{(1,2)}S^2 \SL_2 \oplus S^2\SL_2 \otimes \SL_2   & 2[2][0]   \\   
\bigwedge^{(1,3)}S^2\SL_3\oplus \SL_3\otimes S^2\SL_2  & 2[2,0][2] & &    \\ \hline
\end{array} \] 
\caption{Non-SMF modules of type (N)}
	\label{tab:NonSMF_N}
\end{table}

A graph of type (X) contains two (N) subgraphs (with one of those in mirrored fashion). So in this case we must consider representations with building blocks $(\eta), (\kappa), (\vartheta)$ or $(\sigma)$ . The only possibilities that match together are two times $(\eta)$, $(\eta)$ and $(\kappa)$ or, taking into account $\SO_3 \sim S^2\SL_2$, two times $(\vartheta)$ for $n=1$. But one can easily compute by the Pieri rule, that in each case $\bigwedge^{(1,2)}$ contains multiplicities since
\[ \bigwedge^2 \SL_2\otimes \SL_2 = [0,2] + [2,0] {\rm ~and~} \bigwedge^2 \SL_2\otimes S^2\SL_2 = [0,0] + [0,4] + [2,2]. \] 

Now we turn to diagrams of type (W) and deduce from the above considerations that a SMF diagram of this type has to be of the form
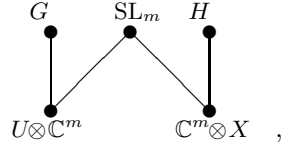
\[ \repW{G}{\SL_m}{H}{U}{\CC^m}{\CC^m}{X}, \]

where $G$ and $H$ are isomorphic to $\SL_n$ or $\SO_{2k+1}$ for $k,n\in\NN$. Hence it must contain subdiagrams 
of type $(\eta)$ or $(\vartheta)$ and the only combination which is distinct from those in Theorem~\ref{thm:SMF_red} corresponds to the action of $\SL_n\times \SL_m \times \SL_k {\rm~on~} \CC^{n(*)}\otimes\CC^m \oplus \CC^m\otimes\CC^{k(*)},~ m>2$. But for $m=3$, $n=k=2$ one has $[1,1,1,1]$ twice in $\bigwedge^{(3,3)}$.

In order to complete the observations of indecomposable representations with two submodules, we have to check the representations \[ S^k\SL_2 \oplus \SL_2\otimes\SL_2\otimes\SL_2 \]
for $k = 1,\dots, 4$, since $\SL_2\otimes \SO_4$ is SMF. But $\bigwedge^3 \SL_2^{\otimes 3}$ decomposes as $[1,1,1] + [1,1,3] + [1,3,1] + [3,1,1]$ and hence $\bigwedge^{(1,3)}$ contains always multiplicities. \\

Now consider representations with three irreducible submodules. Since all of the following modules for $\SL_2$ contain multiplicities in $\bigwedge^{(1,1,1)}$ (as listed below), there is no SMF representation for groups with only one simple factor.

\[ \begin{array}{| l l || l l  |} \hline
\SL_2\oplus\SL_2\oplus \SL_2             & 2[1]  & S^2\SL_2\oplus S^2\SL_2\oplus S^2\SL_2   & 2[4]   \\
\SL_2\oplus\SL_2\oplus S^2\SL_2          & 2[2]  & \SL_2\oplus\SL_2\oplus S^3\SL_2          & 2[3]   \\
\SL_2\oplus S^2\SL_2\oplus S^2\SL_2      & 2[3]  & \SL_2\oplus\SL_2\oplus S^4\SL_2          & 2[4]   \\ \hline
\end{array} \] \bigskip

If there are two simple factors, an indecomposable representation is given by a diagram of type (N) by adding further bottom vertex which must be connected to at least one top vertex. So either there will be a simple factor operating on three irreducible submodules, which is never SMF by the above, or the diagram looks like
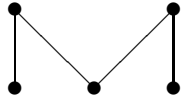
\[ \begin{picture}(100,50)
 \put(15,40){\circle*{5}}   \put(75,40){\circle*{5}}
 \put(15,10){\line(0,1){30}}  \put(75,10){\line(0,1){30}}
 \put(15,40){\line(1,-1){30}}  \put(45,10){\line(1,1){30}} 
 \put(15,10){\circle*{5}}  \put(45,10){\circle*{5}}      \put(75,10){\circle*{5}}
\end{picture}.\]
The simple factors of the group must be isomorphic to $\SL_2$, by comparing this to diagrams of type (N). The eligible modules are $\SL_2 \oplus \SL_2\otimes\SL_2 \oplus \SL_2$, $\SL_2 \oplus \SL_2\otimes\SL_2 \oplus S^2\SL_2$ and $S^2\SL_2 \oplus \SL_2\otimes\SL_2 \oplus S^2\SL_2$. But there is a submodule of multiplicity two in $\bigwedge^4$, namely $[1,1]$, $[1,2]$ and $[2,2]$ respectively. \\

Since the diagram of an module with more than two irreducible components contains always some of the above subdiagrams, there are no further SMF representations and therefore the proof of Theorem~\ref{thm:SMF_red} is complete.

\section*{Acknowledgements}
This work is part of the author's PhD research conducted at University of Erlangen-N\"urnberg under the supervision of Friedrich Knop. The author would like to thank him for his support and suggestions. Furthermore, he thanks Ron C. King, Oksana Yakimova and Guido Pezzini for helpful discussions. Finally, he is indebted to the referees for carefully reading the manuscript. The author appreciates their numerous valuable remarks and the resulting improvements of this article. Some of the research was carried out during a stay at ``Centro di Ricerca Mathematica Ennio De Giorgi'' in Pisa.

\end{document}